\documentclass[a4paper,10pt]{amsart}
\usepackage{amsmath,amsthm,amssymb,latexsym,enumerate,color,hyperref}
\usepackage{graphicx}

\usepackage{xcolor}

\numberwithin{equation}{section}

\begin{document}

\newtheorem{thm}{Theorem}[section]
\newtheorem{prop}[thm]{Proposition}
\newtheorem{lem}[thm]{Lemma}
\newtheorem{cor}[thm]{Corollary}
\newtheorem{rem}[thm]{Remark}
\newtheorem*{defn}{Definition}

\newcommand{\DD}{\mathbb{D}}
\newcommand{\NN}{\mathbb{N}}
\newcommand{\ZZ}{\mathbb{Z}}
\newcommand{\QQ}{\mathbb{Q}}
\newcommand{\RR}{\mathbb{R}}
\newcommand{\CC}{\mathbb{C}}
\renewcommand{\SS}{\mathbb{S}}

\renewcommand{\theequation}{\arabic{section}.\arabic{equation}}

\newcommand{\supp}{\mathop{\mathrm{supp}}}    

\newcommand{\re}{\mathop{\mathrm{Re}}}   
\newcommand{\im}{\mathop{\mathrm{Im}}}   
\newcommand{\dist}{\mathop{\mathrm{dist}}}  
\newcommand{\link}{\mathop{\circ\kern-.35em -}}
\newcommand{\spn}{\mathop{\mathrm{span}}}   
\newcommand{\ind}{\mathop{\mathrm{ind}}}   
\newcommand{\rank}{\mathop{\mathrm{rank}}}   
\newcommand{\Fix}{\mathop{\mathrm{Fix}}}   
\newcommand{\codim}{\mathop{\mathrm{codim}}}   
\newcommand{\conv}{\mathop{\mathrm{conv}}}   
\newcommand{\epsi}{\mbox{$\varepsilon$}}
\newcommand{\eps}{\mathchoice{\epsi}{\epsi}
{\mbox{\scriptsize\epsi}}{\mbox{\tiny\epsi}}}
\newcommand{\cl}{\overline}
\newcommand{\pa}{\partial}
\newcommand{\ve}{\varepsilon}
\newcommand{\zi}{\zeta}
\newcommand{\Si}{\Sigma}
\newcommand{\cA}{{\mathcal A}}
\newcommand{\cG}{{\mathcal G}}
\newcommand{\cH}{{\mathcal H}}
\newcommand{\cI}{{\mathcal I}}
\newcommand{\cJ}{{\mathcal J}}
\newcommand{\cK}{{\mathcal K}}
\newcommand{\cL}{{\mathcal L}}
\newcommand{\cN}{{\mathcal N}}
\newcommand{\cR}{{\mathcal R}}
\newcommand{\cS}{{\mathcal S}}
\newcommand{\cT}{{\mathcal T}}
\newcommand{\cU}{{\mathcal U}}
\newcommand{\OM}{\Omega}
\newcommand{\B}{\bullet}
\newcommand{\ol}{\overline}
\newcommand{\ul}{\underline}
\newcommand{\vp}{\varphi}
\newcommand{\AC}{\mathop{\mathrm{AC}}}   
\newcommand{\Lip}{\mathop{\mathrm{Lip}}}   
\newcommand{\es}{\mathop{\mathrm{esssup}}}   
\newcommand{\les}{\mathop{\mathrm{les}}}   
\newcommand{\nid}{\noindent}
\newcommand{\pzr}{\phi^0_R}
\newcommand{\pir}{\phi^\infty_R}
\newcommand{\psr}{\phi^*_R}
\newcommand{\pow}{\frac{N}{N-1}}
\newcommand{\ncl}{\mathop{\mathrm{nc-lim}}}   
\newcommand{\nvl}{\mathop{\mathrm{nv-lim}}}  
\newcommand{\la}{\lambda}
\newcommand{\La}{\Lambda}    
\newcommand{\de}{\delta}    
\newcommand{\fhi}{\varphi} 
\newcommand{\ga}{\gamma}    
\newcommand{\ka}{\kappa}   

\newcommand{\core}{\heartsuit}
\newcommand{\diam}{\mathrm{diam}}

\newcommand{\lan}{\langle}
\newcommand{\ran}{\rangle}
\newcommand{\tr}{\mathop{\mathrm{tr}}}
\newcommand{\diag}{\mathop{\mathrm{diag}}}
\newcommand{\dv}{\mathop{\mathrm{div}}}

\newcommand{\al}{\alpha}
\newcommand{\be}{\beta}
\newcommand{\Om}{\Omega}
\newcommand{\na}{\nabla}

\newcommand{\cC}{\mathcal{C}}
\newcommand{\cM}{\mathcal{M}}
\newcommand{\nr}{\Vert}
\newcommand{\De}{\Delta}
\newcommand{\cX}{\mathcal{X}}
\newcommand{\cP}{\mathcal{P}}
\newcommand{\om}{\omega}
\newcommand{\si}{\sigma}
\newcommand{\te}{\theta}
\newcommand{\Ga}{\Gamma}

\title[Stability for Serrin's problem and Alexandrov's SBT]{Nearly optimal stability for Serrin's problem and the Soap Bubble theorem}

\author{Rolando Magnanini} 
\address{Dipartimento di Matematica ed Informatica ``U.~Dini'',
Universit\` a di Firenze, viale Morgagni 67/A, 50134 Firenze, Italy.}
    \email{magnanini@unifi.it}
    \urladdr{http://web.math.unifi.it/users/magnanin}

\author{Giorgio Poggesi}
\address{Department of Mathematics and Statistics, The University of Western Australia, 35 Stirling Highway, Crawley, Perth, WA 6009, Australia}
    \email{giorgio.poggesi@uwa.edu.au}

%
%


\begin{abstract}
We present new quantitative estimates for the radially symmetric configuration concerning Serrin's overdetermined problem for the torsional rigidity, Alexandrov's Soap Bubble Theorem, and other related problems. The new estimates improve on those obtained in \cite{MP, MP2} and are in some cases optimal.
\end{abstract}


\maketitle

\raggedbottom

\section{Introduction}
Serrin's symmetry result for the torsional rigidity (\cite{Se}) states that  the overdetermined boundary value problem 
\begin{eqnarray}
\label{serrin1}
&\De u=N \ \mbox{ in } \ \Om, \quad u=0 \ \mbox{ on } \ \Ga, \\
\label{serrin2}
&u_\nu=R \ \mbox{ on } \ \Ga, 
\end{eqnarray}
admits a solution 
for some positive constant $R$ if and only if $\Om$ is a ball of radius $R$ and, up to translations, $u(x)=(|x|^2-R^2)/2$. Here, $\Om$ denotes a bounded domain in $\RR^N$, $N\ge 2$, with sufficiently smooth boundary $\Ga$, say $C^2$, and $u_\nu$ is the outward normal derivative of $u$ on $\Ga$. 

Alexandrov's Soap Bubble Theorem (\cite{Al1}, \cite{Al2}) states that if the mean curvature $H$ of a compact hypersurface $\Ga$ embedded in $\RR^N$ is constant, then $\Ga$ must be a sphere.

In the present paper we consider the stability issue for those two pioneering symmetry theorems.
Technically speaking, we will find two concentric balls $B_{\rho_i} (z)$ and $B_{\rho_e} (z)$, centered at $z \in \Om$ with radii $\rho_i$ and $\rho_e$, such that
%
%
\begin{equation*}
B_{\rho_i} (z) \subseteq \Om \subseteq B_{\rho_e} (z)
\quad \mbox{and} \quad
\rho_e-\rho_i\le \psi(\eta),
\end{equation*}
where
$\psi:[0,\infty)\to[0,\infty)$ is a continuous function vanishing at $0$ and $\eta$ is a suitable measure of the deviation of $u_\nu$ or $H$ from being a constant.
The landmark results of the present paper are the following new stability estimates:
\begin{equation}\label{intro:eq:Improved-Serrin-stability}
\rho_e-\rho_i\le C \, \nr u_\nu - R \nr_{2,\Ga}^{\tau_N} 
\end{equation}
and
\begin{equation}
\label{intro:eq:SBT-improved-stability}
\rho_e-\rho_i\le C\,\nr H_0-H\nr_{2,\Ga}^{\tau_N} .
\end{equation}

In \eqref{intro:eq:Improved-Serrin-stability} (see Theorem \ref{thm:Improved-Serrin-stability} for details), $\tau_2 = 1$, $\tau_3$ is arbitrarily close to one,  
%
%
and $\tau_N = 2/(N-1)$ for $N \ge 4$.
In \eqref{intro:eq:SBT-improved-stability} (see Theorem \ref{thm:SBT-improved-stability} for details), $\tau_N=1$ for $N=2, 3$, $\tau_4$ is arbitrarily close to one, and
%
%
$\tau_N = 2/(N-2) $ for $N\ge 5$.

The constants $C$ depend on the dimension $N$, the diameter $d_\Om$, and the radii $r_i$, $r_e$ of the uniform interior and exterior sphere conditions. 
The dependence on $r_e$ can be removed when $\Ga$ is mean convex.
%
%
%
%
%
%
%

The new estimate \eqref{intro:eq:Improved-Serrin-stability} improves (for every $N \ge 2$) on \cite[Theorem 1.1]{MP2} -- where \eqref{intro:eq:Improved-Serrin-stability} was obtained with $\tau_N= 2/(N+2)$, for every $N \ge 2$ -- to the extent that it gains the (optimal) Lipschitz stability in the case $N=2$.
The estimates obtained in \cite{MP2} were already better than those obtained previously in \cite{ABR, CMV, BNST}. Optimal stability for Serrin's problem has been obtained in \cite{Fe}, but based on a weaker measure of closeness to spherical symmetry. A more detailed overview and comparison of those results can be found in \cite{MP2, Mag, PogTesi}.

%
%

The new estimate \eqref{intro:eq:SBT-improved-stability} 
%
%
improves (for every $N \ge 4$) on \cite[Theorem 1.2]{MP2}, where \eqref{intro:eq:SBT-improved-stability} was obtained with $\tau_N=1$ for $N=2,3$, and $\tau_N=2/(N+2)$ for $N \ge 4$.
If we compare the exponents in \eqref{intro:eq:SBT-improved-stability} 
%
%
to those obtained in \cite[Theorem 1.2]{MP2}, we notice that the dependence of $\tau_N$ on $N$ has become virtually continuous, in the sense that $\tau_N\to 1$, if $N$ ``approaches'' $4$ from below or from above.
%
%
\par
We refer to \cite{MP, MP2, Mag, PogTesi} for a more detailed overview on other stability results present in the literature for Alexandrov's theorem.
%
%
Here, we only comment and compare the optimal Lipschitz stability (i.e. with $\tau_N=1$) for Alexandrov's theorem obtained in general dimension in \cite{CV}, \cite[Theorem 4.6]{MP2}, and \cite{KM}. \cite[Theorem 4.6]{MP2} is based on the same weaker measure of closeness to spherical symmetry used in \cite{Fe} for Serrin's problem. \cite{CV} and \cite[Theorem 1.9]{KM} assume the stronger uniform measure for the deviation of $H$ from being constant. \cite[Theorem 1.10]{KM} holds for surfaces that are $C^1$-small normal deformations of spheres. 
\par
At present, we do not know if the estimate \eqref{intro:eq:SBT-improved-stability} (or \eqref{intro:eq:Improved-Serrin-stability}) is optimal in general dimension. In fact, even if the exponent $1$ obtained in low dimensions is optimal --- as can be verified by direct computations for ellipsoids --- a proof of optimality (or of non-optimality) in higher dimensions is still elusive.

%
%
%

In this paper we will also consider weaker 
%
%
deviations (see Theorems \ref{thm:Serrin-stability} and \ref{thm:SBT-stability}).
Theorem \ref{thm:Serrin-stability} provides the inequality
\begin{equation}
\label{stability-Serrin}
\rho_e-\rho_i\le C\,\nr u_\nu-R\nr_{1,\Ga}^{\tau_N/2},
\end{equation}
where $\tau_N$ is the same appearing in \eqref{intro:eq:Improved-Serrin-stability}.
Also this inequality is new and refines one stated in \cite[Theorem 3.6]{MP2}
in which $\tau_N /2$ was replaced by $1/(N+2)$.

%
%
%
%
In Theorem \ref{thm:SBT-stability} we prove the inequality
\begin{equation}
\label{the-estimate}
\rho_e-\rho_i\le C\,\left\{\int_{\Ga}(H_0-H)^+\,dS_x\right\}^{\tau_N /2} , 
\end{equation}
where $\tau_N$ is the same appearing in \eqref{intro:eq:SBT-improved-stability}.
The last estimate improves (for every $N \ge 4$) on that obtained in \cite[Theorem 4.1]{MP}, where \eqref{the-estimate} was obtained with $\tau_N=1$ for $N=2,3$, and $\tau_N=2/(N+2)$ for $N \ge 4$.

%
%

%
%
All the aforementioned estimates are based on two integral identities obtained in \cite{MP,MP2}:
\begin{equation}
\label{idwps}
\int_{\Om} (-u) \left\{ |\na ^2 u|^2- \frac{ (\De u)^2}{N} \right\} dx=
\frac{1}{2}\,\int_\Ga \left( u_\nu^2- R^2\right) (u_\nu-q_\nu)\,dS_x
\end{equation}
and
\begin{multline}
\label{H-fundamental}
\frac1{N-1}\int_{\Om} \left\{ |\na ^2 u|^2-\frac{(\De u)^2}{N}\right\}dx+
\frac1{R}\,\int_\Ga (u_\nu-R)^2 dS_x = \\ 
\int_{\Ga}(H_0-H)\, (u_\nu)^2 dS_x.
\end{multline}
Here, $R$ and $H_0$ are reference constants given by
\begin{equation}
\label{def-R-H0}
R=\frac{N\,|\Om|}{|\Ga|}, \quad H_0=\frac1{R}=\frac{|\Ga|}{N\,|\Om|} ,
\end{equation}
and $q$ is a quadratic polynomial of the form
\begin{equation}
\label{quadratic}
q(x)=\frac12\, (|x-z|^2-a),
\end{equation}
for some choice of $z\in\RR^N$ and $a\in\RR$.

Identities \eqref{idwps} and \eqref{H-fundamental} hold regardless of how the point $z$ or the constant $a$ are chosen.
%
%
Identity \eqref{idwps}, 
%
%
proved in \cite[Theorem 2.1]{MP2}, puts together and refines Weinberger's argument for symmetry \cite{We} and some remarks of 
%
%
Payne and 
%
%
Schaefer \cite{PS}. Identity \eqref{H-fundamental} was proved in \cite[Theorem 2.2]{MP} by polishing the arguments contained in \cite{Re1} (see also \cite{Re2}).

\par 

The term in the braces in \eqref{idwps} and \eqref{H-fundamental}, that we call Cauchy-Schwarz deficit, plays the role of spherical detector. In fact, by Cauchy-Schwarz inequality, we have that
\begin{equation}
\label{eq:intronewton}
(\De u)^2\le N\,|\na^2 u|^2 \quad \text{in } \Om,
\end{equation}
and the equality sign is identically attained in $\Om$ if and only if 
$u$ is a quadratic polynomial of the form \eqref{quadratic}, and hence $\Ga$ is a sphere centered at $z$, by the boundary condition in \eqref{serrin1}.

It is thus evident that each of the two identities gives spherical symmetry if respectively $u_\nu=R$ or $H=H_0$ on $\Ga$, since Newton's inequality \eqref{eq:intronewton} holds with the equality sign (notice that in \eqref{idwps} $-u>0$ by the strong maximum principle). The same conclusion is also achieved if we only assume that $u_\nu$ or $H$ are constant on $\Ga$, since those constants must equal $R$ and $H_0$ by the identities
\begin{equation*}
\int_\Ga u_\nu\,dS_x=N\,|\Om| \quad \text{and} \quad \int_\Ga H\,q_\nu\,dS_x=|\Ga|.
\end{equation*}
%
%

Thus, \eqref{idwps} and \eqref{H-fundamental} give new elegant proofs of Alexandrov's and Serrin's results.
Moreover, they lead to several advantages and generalizations that have been discussed in \cite{MP, MP2, Pog, PogTesi, Mag}.
The greatest benefits yielded by \eqref{idwps} and \eqref{H-fundamental} are undoubtedly the optimal or quasi-optimal stability results for the Soap Bubble Theorem, Serrin's problem, and other related overdetermined problems.


\par
In order to prove \eqref{intro:eq:Improved-Serrin-stability} and \eqref{intro:eq:SBT-improved-stability}, the harmonic function $h=q-u $ plays a key role. This fact becomes visible when we observe that
\begin{equation}\label{L2-norm-hessian}
|\na ^2 u|^2-\frac{(\De u)^2}{N} = |\na^2 h|^2 
\end{equation}
and, if we choose $z$ in $\Om$,
\begin{equation}\label{oscillation}
\max_{\Ga} h-\min_{\Ga} h=\frac12\,(\rho_e^2-\rho_i^2)\ge \frac{r_i}{2} \, (\rho_e-\rho_i).
\end{equation}
In \eqref{oscillation} we used that $h=q$ on $\Ga$.

%
%
%

Now, since \eqref{idwps} and \eqref{H-fundamental} hold regardless of the choice of the parameters $z$ and $a$ defining $q$, we will thus complete the first step of our argument
%
%
by choosing $z\in\Om$ in a way such that the oscillation of $h$ on $\Ga$ is bounded in terms of the integrals
\begin{equation}\label{eq:introintegrals}
\int_{\Om} \de_\Ga \, |\na ^2 h|^2 \, dx \quad \mbox{or} \quad \frac1{N-1}\int_{\Om} |\na ^2 h|^2 dx ,
\end{equation}
where $\de_\Ga$ denotes the distance to $\Ga$.
In fact, we know that the factor $-u$ appearing in the left-hand side of \eqref{idwps} can be bounded from below by the function $\de_\Ga$, by means of the following inequality proved in \cite[Lemma 3.1]{MP2}:
%
%
\begin{equation}
\label{eq:relationdist}
-u(x) \ge \frac{r_i}{2}\,\de_\Ga(x)\ \mbox{ for every } \ x\in\ol{\Om}.
\end{equation}

%
%

We describe how this task is accomplished. First, as done in Lemmas \ref{lem:Lp-estimate-oscillation-generic-v} and \ref{lem:L2-estimate-oscillation}, we show that the oscillation of $h$ on $\Ga$, and hence $\rho_e - \rho_i$ can be bounded from above in the following way: 
\begin{equation}\label{intro:eq:lemmanuovooscillationLpnorm}
\rho_e - \rho_i \le C (N, p, d_\Om, r_i, r_e) \, \nr h - h_\Om \nr_{p, \Om}^{p/(N+p)} ,
\end{equation}
where $h_\Om$ is the mean value of $h$ on $\Om$ and $p \in \left[1, \infty \right)$.
We stress that this inequality is new and generalizes to any $p$ the estimate obtained in \cite[Lemma 3.3]{MP} for $p=2$.
%
%

Next, to relate the right-hand side of \eqref{intro:eq:lemmanuovooscillationLpnorm} to the integrals in \eqref{eq:introintegrals} we proceed as follows.
We choose $z \in \Om$ as a global minimum 
%
%
point of $u$ (notice that this is always attained in $\Om$) and we apply two integral inequalities to $h$ and its first (harmonic) derivatives. One is the Hardy-Poincar\'e-type inequality 
\begin{equation}
\label{boas-straube}
\nr v \nr_{r, \Om} \le C(N , r, p, \al, d_\Om, r_i) \, \nr \de_\Ga^\al \, \na v \nr_{p, \Om},
\end{equation}
that is applied to the first (harmonic) derivatives of $h$. It holds for any harmonic function $v$
in $\Om$ that is zero at the point $z$ (notice that our choice of $z$ guarantees $\na h(z)=0$). The three numbers $r, p, \al$ are such that $0 \le \al \le 1$ and either $1 \le p \le r \le \frac{Np}{N-p(1 - \al )}$, $p(1 - \al)<N$ ,
or $1 \le r=p <\infty$ (see Lemma \ref{lem:John-two-inequalities} and Remark \ref{rem:nuovaperdirecheBSgeneralizza}).
The other one is applied to
$h- h_\Om$ and is the Poincar\'e-type inequality
\begin{equation}
\label{classicalpoincare}
\nr v \nr_{r, \Om} \le C(N, r, p, d_\Om, r_i) \, \nr \na v \nr_{p, \Om} ,
\end{equation}
that holds for any function $v\in W^{1,p}(\Om)$ with zero mean value on $\Om$. This holds for $r$ and $p$ as above, 
%
%
with $\al=0$ (see Lemma \ref{lem:John-two-inequalities}).

%
%

If we put together \eqref{intro:eq:lemmanuovooscillationLpnorm}-\eqref{classicalpoincare} and choose $\al=1/2$ and $\al=0$ in \eqref{boas-straube} we obtain respectively that
\begin{equation}
\label{the-estimate-for-h}
\rho_e - \rho_i \le C(N, d_\Om, r_i, r_e) \, \left(\int_{\Om} (-u)\, |\na ^2 h|^2\,dx\right)^{\tau_N /2},
\end{equation}
with $\tau_N$ as in \eqref{intro:eq:Improved-Serrin-stability},
and
\begin{equation}
\label{eq:step2}
\rho_e - \rho_i \le C(N, d_\Om, r_i, r_e ) \, \| \na^2 h \|^{\tau_N}_{2, \Om},
\end{equation}
with $\tau_N$ as in \eqref{intro:eq:SBT-improved-stability}.
%
%
%
%
All the details about \eqref{the-estimate-for-h} and \eqref{eq:step2} can be found in Theorems \ref{thm:serrin-W22-stability} and \ref{thm:SBT-W22-stability}.

We mention that in low dimensions -- $N=2$ for \eqref{the-estimate-for-h} and $N=2,3$ for \eqref{eq:step2} -- there is no need to use \eqref{intro:eq:lemmanuovooscillationLpnorm}, thanks to the Sobolev imbedding theorem (see item (i) of Theorems \ref{thm:serrin-W22-stability} and \ref{thm:SBT-W22-stability}).


The proofs of \eqref{intro:eq:Improved-Serrin-stability} and \eqref{intro:eq:SBT-improved-stability} are then completed by the inequalities:
\begin{equation*}
\int_\Om (-u) |\na^2 h|^2 dx \le C(N, d_\Om, r_i, r_e ) \,\nr u_\nu-R\nr_{2,\Ga}^2 ,
\end{equation*}
\begin{equation*}
\| \na ^2 h \|_{2,\Om} \le C (N, d_\Om, r_i, r_e ) \, \| H_0 - H \|_{2, \Ga},
\end{equation*}
that can be deduced by working on the right-hand side of \eqref{idwps} and \eqref{H-fundamental} with arguments taken from \cite{MP2} (see the proof of Theorems \ref{thm:Improved-Serrin-stability} and \ref{thm:SBT-improved-stability} respectively).

\section{Some estimates for harmonic functions}
\label{sec:estimates for harmonic}

%
%
%
We begin by setting some relevant notations. 

By $\Om\subset\RR^N$, $N\ge 2$, we shall denote a bounded domain, that is a connected bounded open set, and call $\Ga$ its boundary. 
By $|\Om|$ and $|\Ga|$, we will denote indifferently the $N$-dimensional Lebesgue measure of $\Om$
and the surface measure of $\Ga$. When $\Ga$ is of class $C^1$, $\nu$ will denote the (exterior) unit normal vector field to $\Ga$ and, when $\Ga$ is a hypersurface of class $C^2$, $H(x)$ will denote its mean curvature  (with respect to $-\nu(x)$) at $x\in\Ga$. 
\par
As already mentioned in the introduction, the diameter of $\Om$ is indicated by $d_\Om$, while $\de_\Ga (x)$ denotes the distance of a point $x$ to the boundary $\Ga$.

We will also use the letter $q$ to denote the quadratic polynomial defined in \eqref{quadratic}, where $z$ is any point in $\RR^N$ and $a$ is any real number; furthermore, we will always use the letter $h$ to denote the harmonic function
$$
h=q-u,
$$
where $u$ is the solution of \eqref{serrin1} and $q$ is the quadratic polynomial defined in \eqref{quadratic}.

For a point $z\in\Om$, $\rho_i$ and $\rho_e$ shall denote 
the radius of the largest ball contained in $\Om$
and that of the smallest ball that contains $\Om$, both centered at $z$; in formulas, 
\begin{equation}
\label{def-rhos}
\rho_i=\min_{x\in\Ga}|x-z|  \ \mbox{ and } \ \rho_e=\max_{x\in\Ga}|x-z|.
\end{equation}
Notice that $\rho_i = \de_\Ga(z)$.

%
%

We recall that if $\Ga$ is of class $C^2$, $\Om$ has the properties of the uniform interior and exterior sphere condition, whose respective radii we have designated by $r_i$ and $r_e$. In other words,
there exists $r_e > 0$ (resp. $r_i>0$) such that for each $p \in \Ga$ there exists a ball contained in $\RR^N \setminus \ol{\Om}$ (resp. contained in $\Om$) of radius $r_e$ (resp. $r_i$) such that
its closure intersects $\Ga$ only at $p$.

Also, if $\Ga$ is of class $C^{2}$,
%
%
the unique solution of \eqref{serrin1} is of class at least
$C^{1,\al}(\ol{\Om})$.
%
%
Thus, we can define
\begin{equation}
\label{bound-gradient}
M=\max_{\ol{\Om}} |\na u|=\max_{\Ga} u_\nu.
\end{equation}
As shown in \cite[Theorem 3.10]{MP}, the following bound holds for $M$:
\begin{equation}
\label{bound-M}
M\le c_N\,\frac{d_\Om(d_\Om+r_e)}{r_e},
\end{equation}
where $c_N=3/2$ for $N=2$ and $c_N=N/2$ for $N\ge 3$. Notice that, when $\Om$ is convex, we can choose $r_e=+\infty$ in \eqref{bound-M} and obtain
\begin{equation}
\label{bound-M-convex}
M\le c_N\,d_\Om.
\end{equation}
More in general, up to a change of the constant $c_N$, \eqref{bound-M-convex} still holds true if $\Ga$ is a mean convex (i.e., $H\ge0$) surface. This is a consequence of \cite[Lemma 2.2]{CM} and a trivial bound. 
Notice that even if \cite[Lemma 2.2]{CM} is stated for strictly mean convex surfaces (i.e., $H>0$), the same proof still works under the weaker assumption $H\ge0$.

For other similar estimates present in the literature, see \cite[Remark 3.11]{MP}.

For a set $A$ and a function $v: A \to \RR$, $v_A$ denotes the {\it mean value of $v$ in $A$} that is
$$
v_A= \frac{1}{|A|} \, \int_A v \, dx.
$$
Also, for a function $v:\Om \to \RR$ we define
$$
\nr \de_\Ga^\al \, \na v \nr_{p,\Om} = \left( \sum_{i=1}^N \nr \de_\Ga^\al \,  v_i \nr_{p,\Om}^p \right)^\frac{1}{p} \quad \mbox{and} \quad
\nr \de_\Ga^\al \, \na^2 v \nr_{p,\Om} = \left( \sum_{i,j=1}^N \nr \de_\Ga^\al \, v_{ij} \nr_{p,\Om}^p \right)^\frac{1}{p},
$$
for $0 \le \al \le 1$ and $p \in [1, \infty)$.

In the present section we prove the estimates \eqref{the-estimate-for-h} and \eqref{eq:step2}.
In order to fulfill this agenda, 
%
%
we first collect some useful estimates for harmonic functions that have their own interest.
Then,
%
%
we will deduce the desired inequalities for the particular harmonic function $h= q- u$.

The following lemma contains Hardy-Poincar\'e inequalities that can be deduced from the works of Hurri-Syrj\"anen \cite{HS1, HS}. We mention that \cite{HS} was stimulated by the work of Boas and Straube \cite{BS}, which, in turn, improved on a result of Ziemer \cite{Zi}.

In order to state these results, we introduce the notions of $b_0$-John domain and $L_0$-John domain with base point $z \in \Om$. Roughly speaking, a domain is a $b_0$-John domain (resp. a $L_0$-John domain with base point $z$) if it is possible to travel from one point of the domain to another (resp. from $z$ to another point of the domain) without going too close to the boundary.

A domain $\Om$ in $\RR^N$ is a {\it $b_0$-John domain}, $b_0 \ge 1$, if each pair of distinct points $a$ and $b$ in $\Om$ can be joined by a curve $\ga: \left[0,1 \right] \rightarrow \Om$ such that
%
%
\begin{equation*}
\de_\Ga (\ga(t)) \ge b_0^{-1} \min{ \left\lbrace |\ga(t) - a|, |\ga(t) - b| \right\rbrace  }.
\end{equation*}

A domain $\Om$ in $\RR^N$ is a {\it $L_0$-John domain with base point $z \in \Om$}, $L_0 \ge 1$, if each point $x \in \Om$ can be joined to $z$ by a curve $\ga: \left[0,1 \right] \rightarrow \Om$ such that
\begin{equation*}
\de_\Ga (\ga(t)) \ge L_0^{-1} |\ga(t) - x|.
\end{equation*}

It is known that, for bounded domains, the two definitions are quantitatively equivalent (see \cite[Theorem 3.6]{Va}).
%
%
The two notions could be also defined respectively through the so-called {\it $b_0$-cigar} and {\it $L_0$-carrot} properties (see \cite{Va}).

\begin{lem}\label{lem:John-two-inequalities}
Let $\Om \subset \RR^N$ be a bounded $b_0$-John domain, and consider three numbers $r, p, \al$ such that, either
\begin{equation}\label{eq:conditionHS}
1 \le p \le r \le \frac{N\,p}{N-p\,(1 - \al )} , \quad p\,(1 - \al)<N , \quad 0 \le \al \le 1 ,
\end{equation}
or
\begin{equation}\label{eq:conditionBS}
r = p \in \left[ 1, \infty \right) , \quad \al=0.
\end{equation}
Then,

(i) there exists a positive constant $ \mu_{r,p, \al} ( \Om, z ) $, such that
\begin{equation}
\label{John-harmonic-quasi-poincare}
\nr v \nr_{r,\Om} \le \mu_{r, p, \al} ( \Om, z)^{-1} \nr \de_\Ga^{\al} \, \na v  \nr_{p, \Om},
\end{equation}
for every function $v$ which is harmonic in $\Om$ and such that $v(z)=0$;
\par
(ii) there exists a positive constant, $\ol{\mu}_{r, p, \al} (\Om)$ such that
\begin{equation}
\label{John-harmonic-poincare}
\nr v - v_\Om \nr_{r,\Om} \le \ol{\mu}_{r, p, \al} (\Om)^{-1} \nr \de_\Ga^{\al} \, \na v  \nr_{p, \Om},
\end{equation}
for every function $v$ which is harmonic in $\Om$.
\end{lem}
\begin{proof}
In \cite[Theorem 1.3]{HS} and \cite[Theorem 8.5]{HS1} it is proved -- if $r, p, \al$ are as in \eqref{eq:conditionHS} and \eqref{eq:conditionBS}, respectively -- that there exists a constant $c=c(N,\, r, \, p,\, \al, \, \Om)$ such that
\begin{equation}\label{eq:risultatodiHSconr-media}
\nr v - v_{r, \Om} \nr_{r,\Om} \le c \, \nr \de_\Ga^{\al} \, \na v  \nr_{p, \Om} ,
\end{equation}
for every $v \in L^1_{loc}(\Om)$ such that $\de_\Ga^\al \, \na v  \in L^p(\Om)$.
Here, $v_{r,\Om}$ denotes the {\it $r$-mean of $v$ in $\Om$} which is defined -- following \cite{IMW} -- as the unique minimizer of the problem
$$\inf_{\la \in \RR} \nr v - \la \nr_{r,\Om}.$$
Notice that, in the case $r=2$, $v_{2, \Om}$ is the classical mean value of $v$ in $\Om$, i.e. $v_{2,\Om} = v_\Om$,
as can be easily verified.

In order to obtain \eqref{John-harmonic-quasi-poincare} and \eqref{John-harmonic-poincare}, we have just to manipulate the left-hand side of \eqref{eq:risultatodiHSconr-media}. To this aim, we exploit the following inequality
\begin{equation}\label{eq:mediaor-mediauguale}
\nr v - v_A \nr_{p,\Om} \le \left[ 1+ \left( \frac{|\Om|}{|A|} \right)^{\frac{1}{p}} \right] \, \nr v - \la \nr_{p,\Om},
\end{equation}
that holds for every $\la \in \RR$, if $\Om$ is a domain with finite measure, $v \in L^p (\Om)$, and $A \subseteq \Om$ is a set of positive measure. Inequality \eqref{eq:mediaor-mediauguale} can be easily proved as follows.
By H\"older's inequality, we have that 
\begin{equation*}
|v_A - \la| \le \frac1{|A|} \int_A|v-\la| \, dx \le 
|A|^{-1/p} \nr v- \la \nr_{p,A} \le
|A|^{-1/p} \nr v - \la \nr_{p,\Om}.
\end{equation*}
Since $|v_A - \la|$ is constant, we then infer that
$$
\nr v_A-\la\nr_{p,\Om}=|\Om|^{1/p} |v_A - \la| \le \left( \frac{|\Om|}{|A|} \right)^{\frac{1}{p}} \, \nr v- \la \nr_{p,\Om}.
$$
Thus, \eqref{eq:mediaor-mediauguale} follows by an application of the triangular inequality.

By using \eqref{eq:mediaor-mediauguale} with $A= \Om$ and $\la = v_{r, \Om}$, from \eqref{eq:risultatodiHSconr-media} we thus prove \eqref{John-harmonic-poincare} for every $v \in L^1_{loc}(\Om)$ such that $\de_\Ga^\al \, \na v  \in L^p(\Om)$.
The assumption of the harmonicity of $v$ in \eqref{John-harmonic-poincare} clearly gives a better constant.

Inequality \eqref{John-harmonic-quasi-poincare} can be deduced from \eqref{John-harmonic-poincare} by applying \eqref{eq:mediaor-mediauguale} with $A= B_{\de_\Ga (z)} (z)$ and $\la = v_\Om$ and recalling that, since $v$ is harmonic, by the mean value property it holds that
$v(z)= v_{B_{\de_\Ga (z)} (z)}.$
%

The (solvable) variational problems 
%
%
\begin{equation*}
\mu_{r, p, \al} (\Om, z) = 
\min \left\{ \nr \de_\Ga^{\al} \, \na v  \nr_{p, \Om} : \nr v \nr_{r, \Om} = 1, \,\De v =0 \text{ in } \Om, \, v(z) = 0 \right\}
\end{equation*}
and
%
%
\begin{equation*}
\ol{\mu}_{r, p, \al} (\Om) = 
\min \left\{ \nr \de_\Ga^{\al} \, \na v  \nr_{p, \Om} : \nr v \nr_{r, \Om} = 1, \,\De v =0 \text{ in } \Om,  v_{\Om} = 0 \right\}
\end{equation*}
then characterize the two constants. 
\end{proof}

\begin{rem}\label{rem:nuovaperdirecheBSgeneralizza}
{\rm
Notice that, by choosing $r=p$ in \eqref{eq:conditionHS} one has the restriction $p (1-\al) < N$, that does not appear in \eqref{eq:conditionBS} (when $\al =0$). We point out that, as proved in \cite{BS}, if $\Ga$ is locally the graph of a function of class $C^{0,\ol{\al}}$, $0 \le \ol{\al} \le 1$, then \eqref{John-harmonic-quasi-poincare} and \eqref{John-harmonic-poincare} still hold true when $r=p \in \left[ 1, \infty \right)$ and $0 \le \al \le \ol{\al}$ (again without the restriction $p (1-\al) < N$). Anyway, in this paper we do not need this generalization. Also, exploiting the proofs of \cite{HS, HS1} has the benefit of building explicit estimates for the constants $\mu_{r, p, \al} (\Om, z)^{-1}$ and $\ol{\mu}_{r, p, \al} (\Om)^{-1}$ (see Remark \ref{rem:stime mu HS in item ii}).
}
\end{rem}

%
%

From Lemma \ref{lem:John-two-inequalities} we can derive estimates for the derivatives of harmonic functions, as follows.

\begin{cor}\label{cor:JohnPoincareaigradienti}
Let $\Om\subset\RR^N$, $N\ge 2$, be a bounded $b_0$-John domain and let $v$ be a harmonic function in $\Om$. Consider three numbers $r, p, \al$ 
satisfying either \eqref{eq:conditionHS} or \eqref{eq:conditionBS}.
%
%

(i) If $z$ is a critical point of $v$ in $\Om$, then it holds that
\begin{equation*}
\nr \na v \nr_{r, \Om} \le \mu_{r, p, \al} ( \Om, z)^{-1} \nr \de_\Ga^{\al} \, \na^2 v  \nr_{p, \Om}.
\end{equation*}

(ii) If
$$\int_\Om \na v \, dx = 0,$$
%
%
then it holds that
\begin{equation*}
\nr \na v \nr_{r, \Om} \le \ol{\mu}_{r, p, \al} (\Om)^{-1} \nr \de_\Ga^{\al} \, \na^2 v  \nr_{p, \Om}.
\end{equation*}
\end{cor}
\begin{proof}
Since $\na v (z)=0$ (respectively $\int_\Om \na v \, dx = 0,$), we can apply \eqref{John-harmonic-quasi-poincare} (respectively \eqref{John-harmonic-poincare}) to each first partial derivative $v_i$ of $v$, $i=1, \dots, N$. If we raise to the power of $r$ those inequalities and sum over $i=1, \dots, N$, the conclusion easily follows in view of the inequality
$$
\sum_{i=1}^N x_i^{\frac{r}{p}} \le \left( \sum_{i=1}^N x_i \right)^{\frac{r}{p}} 
$$
that holds for every $(x_1, \dots, x_N) \in \RR^N$ with $x_i \ge 0$ for $i=1, \dots, N$, since $r/p \ge 1$.
\end{proof}

\begin{rem}[Tracing the geometric dependence of the constants]\label{rem:stime mu HS in item ii}
{\rm
In this remark, we explain how to trace the dependence on a few geometrical parameters of the constants in the relevant inequalities. 

%
%

(i) When $r, p, \al$ are as in \eqref{eq:conditionHS}, the proof of \cite{HS} gives an explicit upper bound for the constant $c$ appearing in \eqref{eq:risultatodiHSconr-media}, from which, by following the steps of our proof, we can deduce explicit estimates for $\mu_{r,p,\al}(\Om,z)^{-1}$ and $\ol{\mu}_{r, p, \al} (\Om)^{-1}$.
In fact, we easily show that
\begin{equation*}
\ol{\mu}_{r, p, \al} (\Om)^{-1} \le k_{N,\, r, \, p,\, \al} \, b_0^N |\Om|^{\frac{1-\al}{N} +\frac{1}{r} - \frac{1}{p} } ,
\end{equation*}
\begin{equation*}
\mu_{r,p,\al}(\Om,z)^{-1} \le k_{N,\, r, \, p, \,\al} \, \left(\frac{b_0}{\de_\Ga (z)^{\frac{1}{r} }}\right)^N |\Om|^{\frac{1-\al}{N} +\frac{2}{r} - \frac{1}{p} } .
\end{equation*}

A better estimate for $\mu_{r,p,\al}(\Om,z)$ can be obtained for $L_0$-John domains with base point $z$. Since the computations are tedious and technical we refer to \cite[Lemma A.2 in Appendix A]{PogTesi} and here we just report the final estimate, that is,
\begin{equation}\label{eq:estimatemu-r-p-al-generalizingFeappendix}
\mu_{r,p,\al}(\Om,z)^{-1} \le k_{N,r,p,\al} \, L_0^N |\Om|^{\frac{1-\al}{N} +\frac{1}{r} - \frac{1}{p} } .
\end{equation}

(ii) When $r,p, \al$ are as in \eqref{eq:conditionBS}, then the proof of \cite[Theorem 8.5]{HS1} gives 
an explicit upper bound for $\ol{\mu}_{p,p,0} (\Om)^{-1}$, in terms of $b_0$ and $d_{\Om}$ only. We warn the reader that the definition of John domain used there is different from the definitions that we gave in this paper, but it is equivalent in view of \cite[Theorem 8.5]{MarS}. Explicitly, by putting together \cite[Theorem 8.5]{HS1} and \cite[Theorem 8.5]{MarS} one finds that
\begin{equation*}
\ol{\mu}_{p,p,0} (\Om)^{-1} \le k_{N, \, p} \, b_0^{3N(1 + \frac{N}{p})} \, d_\Om.
\end{equation*}

Reasoning as in the proof of \eqref{John-harmonic-quasi-poincare}, from this estimate one can also deduce a bound for $\mu_{p,p,0}(\Om,z)$. In fact, by applying \eqref{eq:mediaor-mediauguale} with $A= B_{\de_\Ga (z)} (z)$ and $\la = v_\Om$ and recalling the mean value property of $v$, from \eqref{John-harmonic-poincare} and the bound for $\ol{\mu}_{p,p,0} (\Om)$, we easily compute that
\begin{equation*}
\mu_{p,p,0} (\Om, z)^{-1} \le k_{N, \, p} \, \left( \frac{|\Om|}{\de_\Ga (z)^N} \right)^\frac{1}{r} \,b_0^{3N(1 + \frac{N}{p})} \, d_\Om.
\end{equation*}

A better estimate for $\mu_{p,p,0}(\Om,z)$ can be obtained for $L_0$-John domains with base point $z$, that is,
\begin{equation}\label{eq:estimatemu-p-0-generalizingFeappendix}
\mu_{p,p, 0}(\Om,z)^{-1} \le k_{N, \, p} \, L_0^{3N(1 + \frac{N}{p})} \, d_\Om .
\end{equation}
Complete computations to obtain \eqref{eq:estimatemu-p-0-generalizingFeappendix} can be found in \cite[Lemma A.4 in Appendix A]{PogTesi}.

(iii) A domain of class $C^{2}$ is obviously a $b_0$-John domain and a $L_0$-John domain with base point $z$ for every $z \in \Om$. In fact, by the definitions, it is not difficult to prove the following bounds
%
%
\begin{equation*}
b_0 \le  \frac{d_\Om}{r_i} ,
\end{equation*}
%
%
\begin{equation*}
L_0 \le \frac{d_\Om}{\min[r_i, \de_\Ga (z)] } .
\end{equation*}

Thus, for $C^2$-domains items (i) and (ii) inform us that:
when $r,p, \al$ are as in \eqref{eq:conditionHS}, we have that
\begin{equation*}
\ol{\mu}_{r, p, \al} (\Om)^{-1} \le k_{N,\, r, \, p,\, \al} \, \left( \frac{d_\Om}{r_i} \right)^N |\Om|^{\frac{1-\al}{N} +\frac{1}{r} - \frac{1}{p} } ,
\end{equation*}
\begin{equation*}
\mu_{r,p,\al}(\Om,z)^{-1} \le k_{N,r,p,\al} \, \left( \frac{d_\Om}{\min[r_i, \de_\Ga (z)] } \right)^N |\Om|^{\frac{1-\al}{N} +\frac{1}{r} - \frac{1}{p} } ;
\end{equation*}
when $r,p, \al$ are as in \eqref{eq:conditionBS}, we have that
\begin{equation*}
\ol{\mu}_{p,p,0} (\Om)^{-1} \le k_{N, \, p} \,  \frac{d_\Om^{3N(1 + \frac{N}{p}) + 1 }  }{r_i^{3N(1 + \frac{N}{p})}  }  ,
\end{equation*}
\begin{equation*}
\mu_{p,p, 0}(\Om,z)^{-1} \le k_{N, \, p} \, \frac{d_\Om^{3N(1 + \frac{N}{p}) + 1 }  }{\min[r_i, \de_\Ga (z)]^{3N(1 + \frac{N}{p})} } .
\end{equation*}
}
\end{rem}

\medskip

The next lemma, that modifies for our purposes an idea of W. Feldman \cite{Fe}, will be useful to bound the right-hand side of \eqref{idwps}. We mention that the proof that we report here comes from \cite{MP2}.
%
%
\begin{lem}[A trace inequality for harmonic functions]
\label{lem:genericv-trace inequality}
Let $\Om\subset\RR^N$, $N\ge 2$, be a bounded domain with boundary $\Ga$ of class $C^2$ and let $v$ be a harmonic function in $\Om$.

(i) If $z$ is a critical point of $v$ in $\Om$, then it holds that
\begin{equation*}
\int_{\Ga} |\na v|^2 dS_x \le \frac{2}{r_i} \left(1+\frac{N}{r_i\, \mu_{2,2, \frac{1}{2} }(\Om,z)^2 } \right)  \int_{\Om} (-u) |\na^2 v|^2 dx.
\end{equation*}

(ii) If 
$$\int_\Om \na v \, dx = 0,$$ 
then it holds that
\begin{equation*}
\int_{\Ga} |\na v|^2 dS_x \le \frac{2}{r_i} \left(1+\frac{N}{r_i\, \ol{\mu}_{2,2, \frac{1}{2} }(\Om)^2 } \right)  \int_{\Om} (-u) |\na^2 v|^2 dx.
\end{equation*}
\end{lem}

\begin{proof}
%
%
We begin with the following differential identity:
\begin{equation*}
\label{diffidimpr}
\dv\,\{v^2 \na u - u \, \na(v^2)\}= v^2 \De u - u \, \De (v^2)= N \, v^2 - 2 u \, |\na v|^2,
\end{equation*}
that holds for any $v$ harmonic function in $\Om$, if $u$ is satisfies\eqref{serrin1}.
Next, we integrate on $\Om$ and, by the divergence theorem, we get:
\begin{equation*}
\label{diffidimprint}
\int_{\Ga} v^2 u_{\nu} \, dS_x = N \int_{\Om} v^2 dx + 2 \int_{\Om} (-u) |\na v|^2 dx.
\end{equation*}
We use this identity with $v=v_{i}$, and hence we sum up over $i=1,\dots, N$ to obtain:
\begin{equation*}
\int_{\Ga} |\na v|^2 u_{\nu} dS_x = N \int_{\Om} |\na v|^2 dx + 2 \int_{\Om} (-u) |\na^2 v|^2 dx.
\end{equation*}
Since the term $u_\nu$ at the left-hand side of the last identity 
%
%
can be bounded from below by $r_i$, by an adaptation of Hopf's lemma (see \cite[Theorem 3.10]{MP}), it holds that
\begin{equation*}
r_i \, \int_{\Ga} |\na v|^2 dS_x \le N \int_{\Om} |\na v|^2 dx + 2 \int_{\Om} (-u) |\na^2 v|^2 dx.
\end{equation*}
Thus, the conclusion follows from this last formula, Corollary \ref{cor:JohnPoincareaigradienti} with $r=p=2$ and $\al=1/2$, and \eqref{eq:relationdist}.
\end{proof}

\medskip

We now single out the key lemma that will produce \eqref{the-estimate-for-h} and \eqref{eq:step2}. It contains an inequality  for the oscillation of a harmonic function $v$ in terms of its $L^p$-norm and of a bound for its gradient.
We point out that the following lemma is new and generalizes the estimates proved and used in \cite{MP, MP2} for $p=2$.

To this aim, we define the {\it parallel set} as
$$
\Om_\si=\{ y\in\Om: \de_\Ga (y) >\si\} \quad \mbox{ for } \quad 0<\si \le r_i.
$$

\begin{lem}
\label{lem:Lp-estimate-oscillation-generic-v}
Let $\Om\subset\RR^N$, $N\ge 2$, be a bounded domain with boundary $\Ga$ of class $C^2$ and let $v$ be a harmonic function in $\Om$ of class $C^1 (\ol{\Om})$. Let $G$ be an upper bound for the gradient of $v$ on $\Ga$.
\par
Then, there exist two constants $a_{N,p}$ and $\al_{N,p}$ depending only on $N$ and $p$ such that if
\begin{equation}
\label{smallness-generic-v}
\nr v - v_{\Om} \nr_{p, \Om} \le \al_{N,p} \, r_{i}^{\frac{N+p}{p}}  G  
\end{equation}
holds, we have that
\begin{equation}
\label{Lp-stability-generic-v}  
\max_{\Ga} v - \min_{\Ga} v \le a_{N,p} \,  G^{ \frac{N}{N+p} } \, \nr v - v_{\Om} \nr_{p, \Om}^{ p/(N+p) }.
\end{equation} 
\end{lem}

\begin{proof}
Since $v$ is harmonic it attains its extrema on the boundary $\Ga$.
Let $x_i$ and $x_e$ be points in $\Ga$ that respectively minimize and maximize $v$ on $\Ga$ and, for 
$$
0<\si \le r_i,
$$
define the two points in $y_i, y_e\in\pa\Om_\si$ by
$y_j=x_j-\si\nu(x_j)$, $j=i, e$. 
\par
By the fundamental theorem of calculus we have that
\begin{equation}\label{eq:prova-TFCI-generic v}
v(x_j)= v(y_j) + \int_0^\si \lan\na v(x_j-t\nu(x_j)),\nu(x_j)\ran\,dt.
\end{equation}
%
%
\par
Since $v$ is harmonic and $y_j\in \overline{ \Om }_\si $, $j=i, e$, we can use the mean value property for the balls with radius $\si$ centered at  $y_j$ and obtain: 
\begin{multline*}
|v(y_j) - v_{\Om}| \le \frac1{|B|\, \si^N}\,\int_{B_\si(y_j)}|v - v_{\Om} |\,dy\le \\
\frac{1}{ \left[ |B|\, \si^N \right]^{1/p} } \, \left[\int_{B_\si(y_j)}|v - v_{\Om} |^p\,dy\right]^{1/p}\le 
\frac1{ \left[ |B|\, \si^N \right]^{1/p} } \, \left[\int_{\Om}|v- v_{\Om}|^p\,dy\right]^{1/p} 
\end{multline*}
after an application of H\"older's inequality and by the fact that $B_\si(y_j) \subseteq \Om$. This and \eqref{eq:prova-TFCI-generic v} then yield that
\begin{equation*}
\max_{\Ga} v - \min_{\Ga} v \le 2 \, \left[  \frac{\nr v - v_{\Om} \nr_{p, \Om} }{ |B|^{1/p} \, \si^{N/p}}+ \si \, G \right] ,
\end{equation*}
for every $0<\si \le r_i$.
Here we used that 
%
%
%
$|\na v|$ attains its maximum on $\Ga$, being $v$ harmonic.
\par
Therefore, by minimizing the right-hand side of the last inequality, we can conveniently choose 
$$
\si=\left(\frac{N\,\nr v - v_{\Om} \nr_{p, \Om} }{ p \, |B|^{1/p}\, G }\right)^{p/(N+p)} 
$$
and obtain \eqref{Lp-stability-generic-v}, if $\si \le r_i$;  \eqref{smallness-generic-v} will then follow. The explicit computation immediately shows that
\begin{equation}
\label{eq:costantia_Nal_Nlemmagenericv}
a_{N,p}= \frac{ 2 (N+p) }{N^{\frac{N}{N+p}} p^{\frac{p}{N+p}}  \,|B|^{\frac{1}{N+p}} }
\quad \mbox{and} \quad \al_{N,p}= \frac{ p }{N} \, |B|^{\frac{1}{p} } .
\end{equation}

Notice that, the fact that \eqref{Lp-stability-generic-v} holds if \eqref{smallness-generic-v} is verified, remains true even if we replace in \eqref{smallness-generic-v} and \eqref{Lp-stability-generic-v} $v_\Om$ by any $\la \in \RR$.
\end{proof}

We now turn back our attention to the harmonic function $h = q-u$,
and by exploiting \eqref{oscillation} we now modify Lemma \ref{lem:Lp-estimate-oscillation-generic-v} to
directly link $\rho_e - \rho_i$ to the $L^p$-norm of $h$.
%
%
We do it in the following lemma which generalizes to the case of any $L^p$-norm \cite[Lemma 3.3]{MP}, that holds for $p=2$.

%
%

\begin{lem}
\label{lem:L2-estimate-oscillation}
Let $\Om\subset\RR^N$, $N\ge 2$, be a bounded domain with boundary of class $C^2$.
Set $h=q-u$, where $u$ is the solution of \eqref{serrin1} and $q$ is any quadratic polynomial as in \eqref{quadratic} with $z\in\Om$. 
\par
Then, there exists a positive constant $C$ such that
\begin{equation}
\label{L2-stability}
\rho_e-\rho_i\le C \, \nr h - h_{\Om} \nr_{p, \Om}^{ p/(N+p) }.
\end{equation}
The constant $C$ depends on $N$, $p$, $d_\Om$, $r_i$, $r_e$. If $\Ga$ is mean convex the dependence on $r_e$ can be removed.
%
%
\end{lem}

\begin{proof}
%
%
%
By direct computations it is easy to check that
\begin{equation*}
| \na h | \le M + d_\Om \quad \mbox{on } \ol{\Om},
%
%
\end{equation*}
where $M$ is the maximum of $|\na u|$ on $\ol{\Om}$, as defined in \eqref{bound-gradient}.
Thus, we can apply Lemma \ref{lem:Lp-estimate-oscillation-generic-v} with $v=h$ and $G= M + d_\Om$.
By means of \eqref{oscillation} we deduce that
\eqref{L2-stability} holds with 
\begin{equation}\label{eq:constantCmaxnuovolemmaoscillation}
C= 2 \, a_{N,p} \, \frac{ (M + d_\Om )^{ \frac{N}{N+p} } }{ r_i } ,
\end{equation} 
if
\begin{equation*}
\nr h - h_{\Om} \nr_{p, \Om} \le \al_{N,p} \, (M + d_\Om) \, r_{i}^{\frac{N+p}{p}} .
\end{equation*}
Here,
$a_{N,p}$ and $\al_{N,p}$ are the constants defined in \eqref{eq:costantia_Nal_Nlemmagenericv}. 
On the other hand, if
\begin{equation*}
\nr h - h_{\Om} \nr_{p, \Om} > \al_{N,p} \, (M + d_\Om) \, r_{i}^{\frac{N+p}{p}} ,
\end{equation*}
it is trivial to check that \eqref{L2-stability} is verified with
$$C= \frac{d_\Om}{ \left[ \al_{N,p} \, (M + d_\Om) \right]^{\frac{p}{N+p}} \, r_{i} }.$$
Thus, \eqref{L2-stability} always holds true if we choose
the maximum between this constant and that in \eqref{eq:constantCmaxnuovolemmaoscillation}. We then can easily see that the following constant will do:
%
%
\begin{equation*}
C= \max{ \left\lbrace 2 \, a_{N,p} , \al_{N,p}^{- \frac{p}{N+p}} \right\rbrace } \, \frac{d_\Om^\frac{N}{N+p} }{r_i} \,  \left( 1 + \frac{M}{d_\Om} \right)^{ \frac{N}{N+p} } .
\end{equation*}
%
%
Now, by means of \eqref{bound-M}, we obtain the constant
\begin{equation*}
C= \max{ \left\lbrace 2 \, a_{N,p} , \al_{N,p}^{- \frac{p}{N+p}} \right\rbrace } \, \frac{d_\Om^\frac{N}{N+p} }{r_i} \,  \left( 1 + c_N \, \frac{d_\Om + r_e}{r_e} \right)^{ \frac{N}{N+p} }.
\end{equation*}
If $\Ga$ is mean convex, the dependence on $r_e$ can be avoided and we can choose
\begin{equation*}
C= \left( 1 + c_N \right)^{ \frac{N}{N+p} } \, \max{ \left\lbrace 2 \, a_{N,p} , \al_{N,p}^{- \frac{p}{N+p}} \right\rbrace } \, \frac{d_\Om^\frac{N}{N+p} }{r_i},
\end{equation*}
in light of \eqref{bound-M-convex}.
\end{proof}

%
%

%
%

For Serrin's overdetermined problem, Theorem \ref{thm:serrin-W22-stability} below will be crucial.
There, we associate the oscillation of $h$, and hence $\rho_e - \rho_i$, with the weighted $L^2$-norm of its Hessian matrix.
%
%
%

To this aim, we now choose the center $z$ of the quadratic polynomial $q$ in \eqref{quadratic} to be a global minimum point of $u$, that is always attained in $\Om$. 
%
%
With this choice we have that $\na h (z) =0$.
%
%
We emphasize that the result that we present here improves (for every $N \ge 2$) the exponents of estimates obtained in \cite{MP2}.

\begin{thm}
\label{thm:serrin-W22-stability} 
Let $\Om\subset\RR^N$, $N\ge 2$, be a bounded domain with boundary $\Ga$ of class $C^{2}$
and $z \in \Om$ be a global minimum point of the solution $u$ of \eqref{serrin1}.
Consider the function  $h=q-u$, with $q$ given by \eqref{quadratic}.

There exists a positive constant $C$ such that
\begin{equation}\label{eq:C-provastab-serrin-W22}
\rho_e-\rho_i\le C\, \nr \de_\Ga^{\frac{1}{2} } \, \na^2 h  \nr_{2,\Om}^{\tau_N} ,
\end{equation}
with the following specifications:
%
%
\begin{enumerate}[(i)]
\item $\tau_2 = 1$;
\item $\tau_3$ is arbitrarily close to one, in the sense that for any $\theta>0$, there exists a positive constant $C$ such that  \eqref{eq:C-provastab-serrin-W22} holds with $\tau_3 = 1- \theta$;
\item $\tau_N = 2/(N-1)$ for $N \ge 4$.
\end{enumerate}

The constant $C$
%
%
depends on $N$, $r_i$, $r_e$, $d_\Om$,
and $\theta$ (only in the case $N=3$).
If $\Ga$ is mean convex the dependence on $r_e$ can be removed.
\end{thm}

\begin{proof}
For the sake of clarity, we will always use the letter $c$ to denote the constants in all the inequalities appearing in the proof.
Their explicit computation will be clear by following the steps of the proof.

(i) 
Let $N=2$.
By the Sobolev immersion theorem (for instance we apply \cite[Theorem 9.1]{Fr} to $h-h_\Om$), we deduce that there is
a constant $c$ such that,
\begin{equation}
\label{eq:immersionSerrinN2VERSIONENEW}
\max_{\ol{\Om}} | h - h_\Om |  \le c \, \nr h-h_\Om \nr_{W^{1,4}(\Om)} .
\end{equation}
%
%
%
Applying \eqref{John-harmonic-poincare} with $v=h$, $r=p=4$, and $\al=0$ leads to
$$
\nr h - h_{\Om} \nr_{W^{1,4}(\Om)}\le c \, \nr \na h\nr_{4,\Om}.
$$
Since $\na h(z)=0$, we can apply item (i) of Corollary \ref{cor:JohnPoincareaigradienti} with $r=4$, $p=2$, and $\al=1/2$ to $h$ and obtain that
$$
\nr \na h \nr_{4,\Om} \le c \,  \nr \de_\Ga^{\frac{1}{2} } \, \na^2 h  \nr_{2,\Om} .
$$
Thus, we have that
$$
\nr h - h_{\Om} \nr_{W^{1,4}(\Om)}\le c \, \nr \de_\Ga^{\frac{1}{2} } \, \na^2 h \nr_{2,\Om} .
$$

By using the last inequality together with \eqref{eq:immersionSerrinN2VERSIONENEW}
%
%
we have that
\begin{equation*}
\max_\Ga h-\min_\Ga h \le 
c \, \nr \de_\Ga^{\frac{1}{2} } \, \na^2 h  \nr_{2,\Om}.
\end{equation*}
Thus, by recalling \eqref{oscillation} we get that \eqref{eq:C-provastab-serrin-W22} holds with $\tau_2=1$.

(ii) Let $N=3$.
By applying to the function $h$ \eqref{John-harmonic-poincare} with $r= \frac{3(1- \theta)}{\theta}$,
$p=3(1 -\theta )$, $\al=0$, and item (i) of Corollary \ref{cor:JohnPoincareaigradienti} with
$r=3 ( 1 - \theta )$, $p=2$, $\al=1/2$, we get
$$
\nr h - h_{\Om} \nr_{\frac{3( 1 - \theta)}{\theta}, \Om} \le c \, \nr \de_\Ga^{\frac{1}{2} } \, \na^2 h \nr_{2,\Om}.
$$
Thus, by recalling Lemma \ref{lem:L2-estimate-oscillation} we have that \eqref{eq:C-provastab-serrin-W22} holds true with $\tau_3 = 1- \theta $.

(iii) Let $N \ge 4$. Since $\na h(z)=0$, we can apply to $h$ item (i) of Corollary \ref{cor:JohnPoincareaigradienti} with $r=\frac{2N}{N-1}$, $p=2$, $\al=1/2$, and obtain that
\begin{equation*}
\nr \na h \nr_{\frac{2N}{N-1},\Om} \le c \, \nr \de_\Ga^{\frac{1}{2} } \, \na^2 h  \nr_{2,\Om} .
\end{equation*}
Being $N \ge 4$, we can apply \eqref{John-harmonic-poincare} with $v=h$, $r=\frac{2N}{N-3}$, $p=\frac{2N}{N-1}$, $\al=0$, and get
\begin{equation*}
\nr h- h_{ \Om} \nr_{\frac{2N}{N-3} } \le c \, \nr \na h \nr_{\frac{2N}{N-1},\Om}.
\end{equation*}
Thus,
$$
\nr h- h_{ \Om} \nr_{\frac{2N}{N-3} } \le c \, \nr \de_\Ga^{\frac{1}{2} } \, \na^2 h  \nr_{2,\Om},
$$
and by Lemma \ref{lem:L2-estimate-oscillation} we get that \eqref{eq:C-provastab-serrin-W22} holds with $\tau_N = 2/(N-1)$.
\end{proof}

\begin{rem}[On the constant $C$]\label{rem:dipendenzecostanti}
{\rm
The constant $C$ can be shown to depend only on the parameters mentioned in the statement of Theorem \ref{thm:serrin-W22-stability}. 
%
%
In fact, the parameters $\mu_{r,p, \al} (\Om,z)$ and $\ol{\mu}_{r,p,\al} (\Om)$,  can be estimated by using item (iii) of Remark \ref{rem:stime mu HS in item ii}.
To remove the dependence on the volume, then one can use the trivial bound
$$
|\Om|^{1/N} \le |B|^{1/N} d_\Om /2 .
$$
%
%
\par
The dependence on $\de_\Ga (z)$ can instead be removed by means of the bound
\begin{equation}\label{eq:lowerbounddist}
\de_\Ga (z) \ge \frac{r_i^2}{2M}.
\end{equation}
To prove \eqref{eq:lowerbounddist}, we choose $x\in\Om$ such that $\de_\Ga (x)= r_i$, $y\in\Ga$ such that $\de_\Ga(z)= |y-z|$, and obtain the chain of inequalities
$$
\frac{r_i^2}{2}\le -u(x)\le\max_{\ol{\Om}}(-u) =-u(z)= u(y)-u(z) \le M\, \de_\Ga (z).
$$
Here,
the first inequality follows from \eqref{eq:relationdist} and, as usual, $M$ is defined by \eqref{bound-gradient} and can be estimated with the help of \eqref{bound-M} (or \eqref{bound-M-convex} if $\Ga$ is mean convex). 
We mention that in the convex case an alternative estimate for $\de_\Ga (z)$ could also be obtained by exploiting \cite{BMS}, as explained in \cite[Remark 3.21]{PogTesi}.
\par
We recall that if $\Om$ satisfies the cone property (for the definition see \cite[Chapter 9]{Fr}), the immersion constant (appearing in \eqref{eq:immersionSerrinN2VERSIONENEW})
%
%
depends only on $N$ and the two parameters of the cone property (see \cite[Theorem 9.1]{Fr}).
In our case $\Om$ satisfies the uniform interior sphere condition, and hence the two parameters of the cone property can be easily estimated in terms of $r_i$.
}
\end{rem}

In the case of Alexandrov's Soap Bubble Theorem, we have to deal with \eqref{H-fundamental} or \eqref{identity-SBT-h}, which simply entail the unweighted $L^2$-norm of the Hessian of $h$.
Thus, the appropriate result in this case is Theorem \ref{thm:SBT-W22-stability} below, in which we
%
%
improve (for every $N\ge 4$) the exponents of estimates obtained in \cite{MP}.

\begin{thm}
\label{thm:SBT-W22-stability}
Let $\Om\subset\RR^N$, $N\ge 2$, be a bounded domain with boundary $\Ga$ of class $C^{2}$
and $z \in \Om$ be a global minimum point of the solution $u$ of \eqref{serrin1}.
%
Consider the function  $h=q-u$, with $q$ given by \eqref{quadratic}.
%
%

There exists a positive constant $C$ such that
\begin{equation}
\label{W22-stability}
\rho_e-\rho_i\le C\, \nr \na^2 h\nr_{2,\Om}^{\tau_N} ,
\end{equation}
with the following specifications:
%
%
\begin{enumerate}[(i)]
\item $\tau_N=1$ for $N=2$ or $3$;
\item $\tau_4$ is arbitrarily close to one, in the sense that for any $\theta>0$, there exists a positive constant $C$ such that \eqref{W22-stability} holds with $\tau_4= 1- \theta $;
\item $\tau_N = 2/(N-2) $ for $N\ge 5$.
\end{enumerate}

The constant $C$ depends on $N$, $r_i$, $r_e$, $d_\Om$, and $\theta$ (only in the case $N=4$).
If $\Ga$ is mean convex the dependence on $r_e$ can be removed.
\end{thm}
\begin{proof}
As done in the proof of Theorem \ref{thm:serrin-W22-stability}, for the sake of clarity, we will always use the letter $c$ to denote the constants in all the inequalities appearing in the proof.
Their explicit computation will be clear by following the steps of the proof. By reasoning as described in Remark \ref{rem:dipendenzecostanti}, one can easily check that those constants depend only on the geometric parameters of $\Om$ mentioned in the statement of the theorem.

(i) Let $N=2$ or $3$.
By the Sobolev immersion theorem (for instance we apply \cite[Theorem 9.1]{Fr} to $h-h_\Om$), we deduce that there is a constant $c$ such that,
\begin{equation}
\label{eq:immersionSBTN23VERSIONENEWriscritta}
\max_{\ol{\Om}} | h- h_\Om |  \le c \, \nr h-h_\Om \nr_{W^{2,2}(\Om)} .
\end{equation}
%
%
%
%
Since $\na h(z)=0$, we can apply item (i) of Corollary \ref{cor:JohnPoincareaigradienti} with $r=p=2$ and $\al= 0$ to $h$ and obtain that
\begin{equation*}
\int_{\Om} |\na h|^2 \, dx \le c \, \int_{\Om} | \na^2 h|^2 \, dx.
\end{equation*}
Using this last inequality together with \eqref{John-harmonic-poincare} with $v= h$, $r=p=2$, $\al=0$, leads to
$$
\nr h - h_{\Om} \nr_{W^{2,2}(\Om)}\le c \,\nr \na^2 h\nr_{2,\Om}.
$$

Hence, by using \eqref{eq:immersionSBTN23VERSIONENEWriscritta}
%
%
we get that
\begin{equation*}
\max_\Ga h-\min_\Ga h \le c \, \nr \na^2 h\nr_{2,\Om}.
\end{equation*}
Thus, by recalling \eqref{oscillation} we get that \eqref{W22-stability} holds with $\tau_N=1$.

(ii) Let $N=4$.
By applying \eqref{John-harmonic-poincare} with $r= \frac{4( 1 - \theta)}{\theta}$, $p= 4 ( 1 -\theta )$, $\al=0$, and item (i) of Corollary \ref{cor:JohnPoincareaigradienti} with $r=4 ( 1 - \theta )$, $p=2$, $\al= 0$, for $v=h$ we get
$$
\nr h - h_{ \Om} \nr_{\frac{4( 1 - \theta)}{\theta}, \Om} \le c \, \nr \na^2 h \nr_{2,\Om}.
$$
Thus, by Lemma \ref{lem:L2-estimate-oscillation} we conclude that \eqref{W22-stability} holds with $\tau_4= 1- \theta$.

(iii) Let $N\ge 5$. Applying \eqref{John-harmonic-poincare} with $r= \frac{2N}{N-4}$, $p= \frac{2N}{N-2}$, $\al=0$, and item (i) of Corollary \ref{cor:JohnPoincareaigradienti} with $r={\frac{2N}{N-2}}$, $p=2$, $\al=0$, by choosing $v=h$ we get that
$$
\nr h - h_{ \Om} \nr_{\frac{2N}{N-4},\Om} \le c \, \nr\na^2 h\nr_{2,\Om}.
$$
By Lemma \ref{lem:L2-estimate-oscillation}, we have that \eqref{W22-stability} holds
true with $\tau_N = 2/(N-2)$.
\end{proof}

Other possible choices for the point $z$ are described in the next remark.

\par

\begin{rem}
\label{remarkprova:nuova scelte z}
{\rm
(i) We can choose $z$ as the center of mass of $\Om$. 
In fact, if $z$ is the center of mass of $\Om$, we have that
\begin{multline*}
\int_\Om \na h(x)\,dx=\int_\Om [x-z-\na u(x)]\,dx=\\
\int_\Om x\,dx-|\Om|\,z-\int_\Ga u(x)\,\nu(x)\,dS_x=0.
\end{multline*}
Thus, we can use item (ii) of Corollary \ref{cor:JohnPoincareaigradienti} instead of item (i). In this way, in the estimates of Theorems \ref{thm:serrin-W22-stability} and \ref{thm:SBT-W22-stability} we simply obtain the same constants with $\mu_{r,p, \al}(\Om,z)$ replaced by $\ol{\mu}_{r, p, \al}(\Om)$. 
%
%

It should be noticed that, in this case the extra assumption that $z \in \Om$ is needed, since we want that the ball $B_{\rho_i}(z)$ be contained in $\Om$. 
\par
(ii) Following \cite{Fe}, another possible way to choose $z$ is $z= x_0 - \na u (x_0)$,
where $x_0 \in \Om$ is any point such that $\de_\Ga (x_0) \ge r_i$. In fact, we obtain that $\na h (x_0)=0$ and we can thus use \eqref{John-harmonic-quasi-poincare}, with $\mu_{r,p,\al}(\Om,z)$ replaced by $\mu_{r,p,\al}(\Om,x_0)$.
%
%
\par
As in item (i), we should additionally require that $z \in \Om$, to be sure that the ball $B_{\rho_i}(z)$ be contained in $\Om$.
}
\end{rem}

\section{Stability results}
\label{sec:stability results}

In this section, we collect our results on the stability of the spherical configuration by putting together the identities presented in the introduction and already obtained in \cite{MP,MP2} with the new estimates obtained in Section \ref{sec:estimates for harmonic}.

\subsection{Stability for Serrin's overdetermined problem}\label{subsec:stability Serrin}

%
By using \eqref{L2-norm-hessian}, \eqref{idwps} can be rewritten in terms of the harmonic function $h$ as 
\begin{equation}
\label{idwps-h}
\int_{\Om} (-u)\, |\na ^2 h|^2\,dx=
\frac{1}{2}\,\int_\Ga ( R^2-u_\nu^2)\, h_\nu\,dS_x .
\end{equation}

%
%
%
%
%

\par
In light of \eqref{eq:relationdist}, Theorem \ref{thm:serrin-W22-stability} gives an estimate from below of the left-hand side of \eqref{idwps-h}. Now, we will take care of its right-hand side and prove our main result for Serrin's problem. The result that we present here, improves (for every $N \ge 2$) the exponents in the estimate obtained in \cite[Theorem 1.1]{MP2}.

\begin{thm}[Stability for Serrin's problem]
\label{thm:Improved-Serrin-stability}
Let $\Om\subset\RR^N$, $N\ge2$, be a bounded domain with boundary $\Ga$ of class $C^2$ and $R$ be the constant defined in \eqref{def-R-H0}.
Let $u$ be the solution of problem \eqref{serrin1} and 
$z\in\Om$ be a global minimum point.
\par
There exists a positive constant $C$ such that
\begin{equation}
\label{general improved stability serrin C}
\rho_e-\rho_i\le C\,\nr u_\nu - R \nr_{2,\Ga}^{\tau_N} ,
\end{equation}
with the following specifications:
%
%
%
\begin{enumerate}[(i)]
\item $\tau_2 = 1$;
\item $\tau_3$ is arbitrarily close to one, in the sense that for any $\theta>0$, there exists a positive constant $C$ such that  \eqref{general improved stability serrin C} holds with $\tau_3 = 1- \theta$;
\item $\tau_N = 2/(N-1)$ for $N \ge 4$.
\end{enumerate}

The constant $C$ depends on $N$, $r_i$, $r_e$, $d_\Om$,
and $\theta$ (only in the case $N=3$).
If $\Ga$ is mean convex the dependence on $r_e$ can be removed.
\end{thm}
\begin{proof}
We have that 
$$
\int_\Ga ( R^2-u_\nu^2)\, h_\nu\,dS_x\le (M+R)\,\nr u_\nu-R\nr_{2,\Ga} \nr h_\nu\nr_{2,\Ga},
$$
after an an application of H\"older's inequality.
Thus, by item (i) of Lemma \ref{lem:genericv-trace inequality} with $v=h$, \eqref{idwps-h}, and this inequality, we infer that
\begin{multline*}
\nr h_\nu\nr_{2,\Ga}^2\le \frac{2}{r_i} \left(1+\frac{N}{r_i\, \mu_{2,2,\frac{1}{2} }(\Om,z)^2 } \right)  \int_{\Om} (-u) |\na^2h|^2 dx \le \\ 
\frac{M+R}{r_i} \left(1+\frac{N}{r_i\, \mu_{2,2,\frac{1}{2} }(\Om,z)^2} \right)\nr u_\nu-R\nr_{2,\Ga} \nr h_\nu\nr_{2,\Ga},
\end{multline*}
and hence
\begin{equation}
\label{ineq-feldman}
\nr h_\nu\nr_{2,\Ga}\le \frac{M+R}{r_i} \left(1+\frac{N}{r_i\, \mu_{2,2,\frac{1}{2} }(\Om,z)^2} \right)\nr u_\nu-R\nr_{2,\Ga}.
\end{equation}
Therefore,
\begin{multline*}
\int_\Om |\na ^2 h|^2 \de_\Ga (x)\, dx \le \frac{2}{r_i}\int_\Om (-u) |\na^2 h|^2 dx \le \\
 \left(\frac{M+R}{r_i}\right)^2 \left(1+\frac{N}{r_i \, \mu_{2,2,\frac{1}{2} }(\Om,z)^2} \right)\nr u_\nu-R\nr_{2,\Ga}^2,
\end{multline*}
by \eqref{eq:relationdist}.
%
%
These inequalities and Theorem \ref{thm:serrin-W22-stability} then give the desired conclusion.

We recall that $\mu_{2,2,\frac{1}{2} }(\Om,z)$ appearing in the constant in the last inequality can be estimated in terms of $d_\Om$ and $r_i$
%
%
by proceeding as described in Remark \ref{rem:dipendenzecostanti}.
The ratio $R$ can be estimated (from above) in terms of $|\Om|^{1/N}$ -- just by using the isoperimetric inequality; in turn, $|\Om|^{1/N}$ can be bounded in terms of $d_\Om$ by proceeding as described in Remark \ref{rem:dipendenzecostanti}. 
Finally, as usual, $M$ can be estimated by means of \eqref{bound-M}.
%
%
\end{proof}

If we want to measure the deviation of $u_\nu$ from $R$ in $L^1$-norm, we get a smaller (reduced by one half) stability exponent. The following result improves \cite[Theorem 3.6]{MP2}.

\begin{thm}[Stability with $L^1$-deviation]
\label{thm:Serrin-stability}
Theorem \ref{thm:Improved-Serrin-stability} still holds
with \eqref{general improved stability serrin C} replaced by 
\begin{equation*}
\rho_e-\rho_i\le C\,\nr u_\nu - R \nr_{1,\Ga}^{\tau_N/2} .
\end{equation*}
%
%
%
\end{thm}

\begin{proof}
Instead of applying H\"older's inequality to the right-hand side of \eqref{idwps-h}, we just use the rough bound:
%
%
\begin{equation*}
\int_{\Om} (-u)\, |\na ^2 h|^2\,dx \le
\frac{1}{2}\, \left( M + R \right)\,(M + d_{\Om}) \, \int_\Ga \left| u_\nu - R \right| \, dS_x,
\end{equation*}
since $(u_\nu+R)\,|h_\nu|\le ( M + R)\,(M + d_{\Om})$ on $\Ga$. The conclusion then follows from similar arguments.
\end{proof}

\begin{rem}\label{rem:serrin-stimequantitativefissandoparametri}
{\rm
The estimates presented in Theorems \ref{thm:Improved-Serrin-stability}, \ref{thm:Serrin-stability} may be interpreted as stability estimates, once that we fixed some a priori bounds on the relevant parameters: here, we just illustrate the case of Theorem \ref{thm:Improved-Serrin-stability}. 
Given two positive constants $\ol{d}$ and $\ul{r}$, let ${\mathcal D}={\mathcal D} (\ol{d}, \ul{r} )$ be the class of bounded domains $\Om\subset\RR^N$ with boundary of class $C^{2}$, such that 
$$
d_{ \Om} \le \ol{d}, \quad r_i(\Om), \ r_e(\Om)\ge \ul{r} .
$$
Then, for every $\Om\in {\mathcal D} $, we have that
$$
\rho_e-\rho_i\le C\, \nr u_\nu - R \nr_{2,\Ga}^{\tau_N},
$$
where $\tau_N$ is that appearing in \eqref{general improved stability serrin C} and $C$ is a constant depending on $N$, $\ol{d}$, $\ul{r}$ (and $\theta$ only in the case $N=3$).
\par
If we relax the a priori assumption that $\Om\in {\mathcal D}$ (in particular if we remove the lower bound $\ul{r}$), it may happen that, 
as the deviation $\nr u_\nu - R \nr_{2,\Ga}$ tends to $0$, $\Om$ tends to the ideal configuration of two or more disjoint balls, while $C$ diverges since $\ul{r}$ tends to $0$. The configuration of more balls connected with tiny (but arbitrarily long) tentacles has been quantitatively studied in \cite{BNST}.
}
\end{rem}

\subsection{Stability for Alexandrov's Soap Bubble Theorem}\label{subsec:Alexandrov's SBT}
This section is devoted to the stability issue for the Soap Bubble Theorem.

As we have already noticed in \cite[Theorem 2.4]{MP2}, the right-hand side of \eqref{H-fundamental} can be rewritten as 
%
%
$$
\int_{\Ga}(H_0-H)\,(u_\nu-q_\nu)\,u_\nu\,dS_x+
\int_{\Ga}(H_0-H)\, (u_\nu-R)\,q_\nu\, dS_x.
$$
Hence, by recalling \eqref{L2-norm-hessian}, \eqref{H-fundamental} becomes
\begin{multline}
\label{identity-SBT-h}
\frac1{N-1}\int_{\Om} |\na ^2 h|^2 dx+
\frac1{R}\,\int_\Ga (u_\nu-R)^2 dS_x = \\
-\int_{\Ga}(H_0-H)\,h_\nu\,u_\nu\,dS_x+
\int_{\Ga}(H_0-H)\, (u_\nu-R)\,q_\nu\, dS_x.
\end{multline}

%
%

Next, we derive the following lemma, that parallels and is a useful consequence of Lemma \ref{lem:genericv-trace inequality}.

\begin{lem}
\label{lem:improving-SBT}
Let $\Om\subset\RR^N$, $N\ge2$, be a bounded domain with boundary $\Ga$ of class $C^2$.
%
%
Denote by $H$ the mean curvature of $\Ga$ and let $H_0$ be the constant defined in \eqref{def-R-H0}.
\par
Then, the following inequality holds:
\begin{equation}
\label{improving-SBT}
\nr  u_\nu-R\nr_{2,\Ga} \le 
R\left\{ d_\Om +\frac{M (M+R)}{r_i}\left(1+\frac{N}{r_i \, \mu_{2,2,\frac{1}{2} }(\Om,z)^2 }\right)\right\} \nr H_0-H\nr_{2,\Ga}.
\end{equation}
\end{lem}

\begin{proof}
Discarding the first summand on the left-hand side of \eqref{identity-SBT-h} and applying H\"older's inequality on its right-hand side gives that
$$
\frac1{R} \nr  u_\nu-R\nr_{2,\Ga}^2 \le \nr H_0-H\nr_{2,\Ga}\left( M \nr h_\nu\nr_{2,\Ga} + d_\Om \,\nr  u_\nu-R\nr_{2,\Ga} \right),
$$ 
since $u_\nu\le M$ and $|q_\nu|\le d_\Om $ on $\Ga$. Thus, inequality \eqref{ineq-feldman}
implies that
\begin{multline*}
\nr  u_\nu-R\nr_{2,\Ga}^2 \le 
\\
R\left\{ d_\Om +\frac{M (M+R)}{r_i}\left(1+\frac{N}{r_i \, \mu_{2,2,\frac{1}{2} }(\Om,z)^2}\right)\right\} \nr H_0-H\nr_{2,\Ga} \nr  u_\nu-R\nr_{2,\Ga},
\end{multline*}
from which \eqref{improving-SBT} follows at once.
\end{proof}

We are now ready to prove our main result for Alexandrov's Soap Bubble Theorem. The result that we present here, improves (for every $N \ge 4$) the exponents of estimates obtained in \cite[Theorem 1.2]{MP2}.

\begin{thm}[Stability for the Soap Bubble Theorem] 
\label{thm:SBT-improved-stability}
Let $N\ge 2$ and let $\Ga$ be a surface of class $C^{2}$, which is the boundary of a bounded domain $\Om\subset\RR^N$. Denote by $H$ the mean curvature of $\Ga$ and let $H_0$ be the constant defined in \eqref{def-R-H0}.
\par
Then,
%
%
for some point $z\in\Om$
there exists a positive constant $C$ such that
\begin{equation}
\label{SBT-improved-stability-C}
\rho_e-\rho_i \le C\,\nr H_0-H\nr_{2,\Ga}^{\tau_N} ,
\end{equation}
with the following specifications:
%
%
%
\begin{enumerate}[(i)]
\item $\tau_N=1$ for $N=2$ or $3$;
\item $\tau_4$ is arbitrarily close to one, in the sense that for any $\theta>0$, there exists a positive constant $C$ such that \eqref{SBT-improved-stability-C} holds with $\tau_4= 1- \theta $;
\item $\tau_N = 2/(N-2) $ for $N\ge 5$.
\end{enumerate}

The constant $C$ depends on $N$, $r_i$, $r_e$, $d_\Om$, and $\theta$ (only in the case $N=4$).
If $\Ga$ is mean convex the dependence on $r_e$ can be removed.
\end{thm}
\begin{proof}
As before, we choose $z\in\Om$ to be a global minimum point of $u$.
Discarding the second summand on the left-hand side of \eqref{identity-SBT-h} and applying H\"older's inequality on its right-hand side, as in the previous proof, gives that
\begin{multline*}
\frac1{N-1}\,\int_\Om |\na^2 h|^2 dx\le \\ 
R\left\{ d_\Om +\frac{M (M+R)}{r_i}\left(1+\frac{N}{r_i \, \mu_{2,2,\frac{1}{2} }(\Om,z)^2}\right)\right\} \nr H_0-H\nr_{2,\Ga} \nr  u_\nu-R\nr_{2,\Ga}, \le \\
R^2\left\{ d_\Om +\frac{M (M+R)}{r_i}\left(1+\frac{N}{r_i \, \mu_{2,2,\frac{1}{2} }(\Om,z)^2 }\right)\right\}^2 \nr H_0-H\nr_{2,\Ga}^2,
\end{multline*}
where the second inequality follows from Lemma \ref{lem:improving-SBT}.
\par
The conclusion then follows from
Theorem \ref{thm:SBT-W22-stability}.
\end{proof}

\begin{rem}
{\rm
Estimates similar to those of Theorem \ref{thm:SBT-improved-stability} can also be obtained as a direct corollary of Theorem \ref{thm:Improved-Serrin-stability}, by means of \eqref{improving-SBT}. As it is clear, in this way the exponents $\tau_N$ would be worse than those obtained in Theorem \ref{thm:SBT-improved-stability}.
}
\end{rem}

We now present a stability result with a weaker deviation (at the cost of getting a smaller stability exponent), analogous to Theorem \ref{thm:Serrin-stability}.
To this aim, we notice that from 
\eqref{H-fundamental} and \eqref{L2-norm-hessian} 
%
%
we can easily deduce the following inequality 
%
%
%
\begin{equation}
\label{fundamental-stability2} 
\frac1{N-1}\int_{\Om} |\na^2 h |^2 \,dx + \int_{\Ga}(H_0-H)^-\, (u_\nu)^2\,dS_x \le 
\int_{\Ga}(H_0-H)^+\, (u_\nu)^2\,dS_x
\end{equation}
(here, we use the positive and negative part functions $(t)^+=\max(t,0)$ and $(t)^-=\max(-t,0)$). That inequality tells us that, if we have an {\it a priori} bound $M$ for
$u_\nu$ on $\Ga$, then its left-hand side is small if the integral
\begin{equation}\label{eq:provadeviationpartepositiva}
\int_{\Ga}(H_0-H)^+\,dS_x
\end{equation}
is also small.  
%
%
%
It is clear that \eqref{eq:provadeviationpartepositiva} is a deviation weaker than $\nr H_0-H\nr_{1,\Ga}$.

The result that we present here, improves (for every $N \ge 4$) the estimates obtained in \cite[Theorem 4.1]{MP}.

\begin{thm}[Stability with $L^1$-type deviation]
\label{thm:SBT-stability}
Theorem \ref{thm:SBT-improved-stability} still holds with \eqref{SBT-improved-stability-C} replaced by
\begin{equation*}
\rho_e-\rho_i \le C\,  \left\lbrace \int_\Ga (H_0-H)^+\,dS_x \right\rbrace^{\tau_N /2} .
\end{equation*}
\end{thm}

\begin{proof}
%
%
From \eqref{fundamental-stability2}, we infer that
\begin{equation*}
\nr \na^2 h\nr_{2,\Om}\le M\,\sqrt{N-1}\,  \left\lbrace \int_\Ga (H_0-H)^+\,dS_x \right\rbrace^{1/2}
\end{equation*}
%
%

The conclusion then follows from Theorem \ref{thm:SBT-W22-stability}.
\end{proof}

%
%
%
%
%
%
%
%

\begin{rem}\label{rem:stimequantitativefissandoparametri}
{\rm
The estimates presented in Theorems \ref{thm:SBT-improved-stability}, \ref{thm:SBT-stability} may be interpreted as stability estimates, once some a priori information is available: here, we just illustrate the case of Theorem \ref{thm:SBT-improved-stability}. 
Given two positive constants $\ol{d}$ and $\ul{r}$, let $\cS=\cS(\ol{d}, \ul{r} )$ be the class of connected surfaces $\Ga\subset\RR^N$ of class $C^{2}$, where $\Ga$ is the boundary of a bounded domain $\Om$, such that 
$$
d_{ \Om} \le \ol{d}, \quad r_i(\Om), \ r_e(\Om)\ge \ul{r} .
$$
Then, for every $\Ga\in\cS$ we have that
$$
\rho_e-\rho_i\le C\, \nr H_0-H\nr_{2,\Ga}^{\tau_N},
$$
where $\tau_N$ is that appearing in \eqref{SBT-improved-stability-C} and the constant $C$ depends on $N$, $\ol{d}$, $\ul{r}$ (and $\theta$ only in the case $N=4$).
\par
If we relax the a priori assumption that $\Ga\in\cS$ (in particular if we remove the lower bound
$\ul{r}$), it may happen that, 
as the deviation $\nr H_0 - H\nr_{2,\Ga}$ tends to $0$, $\Om$ tends to the ideal configuration of two or more mutually tangent balls, while $C$ diverges since $\ul{r}$ tends to $0$. Such a configuration can be observed, for example, as limit of sets created by truncating (and then smoothly completing) unduloids with very thin necks. This phenomenon (called bubbling) has been quantitatively studied in \cite{CM} by considering strictly mean convex surfaces and by using the uniform deviation $\nr H_0 - H\nr_{\infty,\Ga}$.
}
\end{rem}

\subsection{Other related stability results}\label{subsec:Other stability results mean curvature}
As already noticed in \cite[Theorem 2.6]{MP}, if $\Ga$ is mean-convex, that is $H \ge 0$, \eqref{H-fundamental} can also be rearranged into the following identity:
\begin{equation}
\label{heintze-karcher-identity}
\frac1{N-1}\,\int_{\Om} \left\{ |\na ^2 u|^2-\frac{(\De u)^2}{N}\right\}\,dx +\int_\Ga\frac{(1-H\,u_\nu)^2}{H}\,dS_x =
\int_\Ga\frac{dS_x}{H}-N |\Om|.
\end{equation}
Such identity leads to Heintze-Karcher's inequality (\cite{HK}):
\begin{equation}
\label{heintze-karcher}
\int_\Ga \frac{dS_x}{H}\ge N |\Om| .
\end{equation}
In fact, since both summands at the left-hand side of \eqref{heintze-karcher-identity} are non-negative,
Heintze-Karcher's inequality \eqref{heintze-karcher} holds and the equality sign is attained if and only if $\Om$ is a ball.
In fact, if the right-hand side of \eqref{heintze-karcher-identity} is zero, both summands at the left-hand side must be zero, and the vanishing of the first summand implies that $\Om$ is a ball, as already noticed.

Thus, by putting together the identity \eqref{heintze-karcher-identity} and the tools developed in Section \ref{sec:estimates for harmonic} we obtain a stability result for Heintze-Karcher's inequality.
%
%
The following theorem improves (for every $N\ge 4$) the exponents obtained in \cite[Theorem 4.5]{MP}.
\begin{thm}[Stability for Heintze-Karcher's inequality]
\label{thm:stability-hk}
Let $\Ga$ be a surface of class $C^2$, which is the boundary of a bounded domain $\Om\subset\RR^N$, $N\ge 2$. Denote by $H$ its mean curvature and suppose that $H\ge0$ on $\Ga$.

Then, there exist a point $z\in\Om$ and a positive constant $C$ such that
\begin{equation}
\label{stability-hk-C}
\rho_e-\rho_i\le C\,\left(\int_\Ga\frac{dS_x}{H}-N\,|\Om|\right)^{\tau_N / 2} ,
\end{equation}
with the following specifications:
\begin{enumerate}[(i)]
\item $\tau_N=1$ for $N=2$ or $3$;
\item $\tau_4$ is arbitrarily close to one, in the sense that for any $\theta>0$, there exists a positive constant $C$ such that \eqref{stability-hk-C} holds with $\tau_4= 1- \theta $;
\item $\tau_N = 2/(N-2) $ for $N\ge 5$.
\end{enumerate}

The constant $C$ depends on $N$, $r_i$, $d_\Om$, and $\theta$ (only in the case $N=4$).
\end{thm}

\begin{proof}
As before, we choose $z\in\Om$ to be a global minimum point of $u$.
Moreover, by \eqref{heintze-karcher-identity} and \eqref{heintze-karcher}, we have that 
$$
\frac1{N-1}\,\int_\Om |\na^2 h|^2\,dx\le \int_\Ga\frac{dS_x}{H}-N\,|\Om|.
$$
Thus, the conclusion follows from Theorem \ref{thm:SBT-W22-stability}.
%
%
%
\end{proof}

Since the deficit of Heintze-Karcher's inequality can be written as
$$
\int_\Ga\frac{dS_x}{H}-N\,|\Om| = \int_\Ga \left( \frac{1}{H} - u_\nu \right) \, dS_x ,
$$
where $u$ is always the solution of \eqref{serrin1},
Theorem \ref{thm:stability-hk} also gives symmetry and stability for the boundary value problem \eqref{serrin1} under the overdetermination 
\begin{equation}
u_\nu = \frac{1}{H},
\end{equation}
as stated next. It is clear that the deviation $\int_\Ga \left( 1/H - u_\nu \right) \, dS_x$ is weaker than $\nr 1/H - u_\nu \nr_{1, \Ga}$. The following theorem improves \cite[Theorem 4.8]{MP}.
 
\begin{thm}[Stability for a related overdetermined problem]
\label{thm:OBVP-stability}
%
Let $\Ga$, $\Om$, and $H$ be as in Theorem \ref{thm:stability-hk}.
%
%
Let $u$ be the solution of problem \eqref{serrin1} and $z \in \Om$ be a global minimum point of $u$.
\par
There exists a positive constant $C$ such that
\begin{equation}
\label{OBVT-stability-gener-C}
\rho_e-\rho_i\le C\, \left\lbrace \int_\Ga \left( \frac{1}{H} - u_\nu \right) \, dS_x \right\rbrace^{ \tau_N / 2 } ,
\end{equation}
%
%
with the following specifications:
\begin{enumerate}[(i)]
\item $\tau_N=1$ for $N=2$ or $3$;
\item $\tau_4$ is arbitrarily close to one, in the sense that for any $\theta>0$, there exists a positive constant $C$ such that \eqref{OBVT-stability-gener-C} holds with $\tau_4= 1- \theta $;
\item $\tau_N = 2/(N-2) $ for $N\ge 5$.
\end{enumerate}

The constant $C$ depends on $N$, $r_i$, $d_\Om$, and $\theta$ (only in the case $N=4$).
\end{thm}

%
%
%
%
%
%

\section*{Acknowledgements}
The authors wish to thank the anonymous referee, who hinted the estimate \eqref{eq:lowerbounddist} and whose suggestions contributed to a  better presentation of this article.
\par
 The paper was partially supported by the Gruppo Nazionale Analisi Matematica Probabilit\`a e Applicazioni (GNAMPA) of the Istituto Nazionale di Alta Matematica (INdAM).

\end{document}